\documentclass[12pt,a4paper]{amsart}
\usepackage{amssymb}
\usepackage{amsfonts}
\usepackage{amsthm}
\usepackage{amsmath}
\usepackage{amscd}
\usepackage[latin2]{inputenc}
\usepackage{t1enc}
\usepackage[mathscr]{eucal}
\usepackage{indentfirst}
\usepackage{graphicx}
\usepackage{graphics}
\numberwithin{equation}{section}
\usepackage[margin=2.9cm]{geometry}
\usepackage{epstopdf} 
\usepackage{xcolor}
 
\usepackage{verbatim}
 \usepackage{enumitem}

\theoremstyle{plain}
\newtheorem{theorem}{Theorem}[section]
\newtheorem{lemma}[theorem]{Lemma}
\newtheorem{corollary}[theorem]{Corollary}
\newtheorem{proposition}[theorem]{Proposition}

 \theoremstyle{definition}
\newtheorem{definition}[theorem]{Definition}

\newtheorem{remark}[theorem]{Remark}
\newtheorem{?}[theorem]{Problem}
\newtheorem{example}[theorem]{Example}

\begin{document}


\title[Boundedness of Fourier multipliers]{ $L_p\rightarrow L_q$ boundedness of Fourier multipliers}

\author[Medet Nursultanov]{M. Nursultanov}
\address {Department of Mathematics and Statistics, University of Helsinki}
\email{medet.nursultanov@gmail.com}



\maketitle

\begin{abstract}
	We investigate the $L_p \mapsto L_q$ boundedness of the Fourier multipliers. We obtain sufficient conditions, namely, we derive Hormander and Lizorkin type theorems. We also obtain the necessary conditions. For $M$-generalized monotone functions, we obtain a criteria for boundedness of the corresponding Fourier multipliers.
\end{abstract}
\section{Introduction}

The study of Fourier multipliers has been attracting attention of researchers for more than a century. This is related to numerous applications in mathematical analysis, in particular, in partial differential equations. One of the important questions in this field is to understand the $L_p\rightarrow L_q$ boundedness of a Fourier multipliers. 

In case $p=q$, one of the earliest important works was obtained by Marcinkiewicz \cite{Marcinkiewicz} in 1939, see also \cite{KaczmarzMarcinkiewicz}. He obtained a sufficient condition for $L_p\rightarrow L_p$ boundedness of Fourier series multipliers. An analogue of his result for Fourier transform multipliers also holds, see \cite{Maligranda}. Another important result was obtained by Mikhlin \cite{Mikhlin} in 1956, which was improved by Stain \cite{Stein} and H\"ormander \cite{Hormander}. There were further developments in this topic, we mention works \cite{CalderonTorchinsky,Grafakos2021,GrafakoSlavikova2019,GrafakoSlavikova2019Asharp,Hytonen2004,NursultanovE1998Mathnotes} and references therein. We also refer to the work \cite{Grafakos2021Someremarks} for a short historical overview of the Mikhlin-H\"ormander and Marcinkiewicz theorems.

For the case $p\leq q$, there another two classical results available: H\"ormander's multiplier theorem \cite{Hormander} and Lizorkin's multiplier theorem \cite{Lizorkin}. There is a fundamental difference between these two results: H\"ormander's theorem does not require any regularity of the symbol and applies to $p$ and $q$ separated by 2, while Lizorkin's theorem requires weaker conditions on $p$, $q$ but imposes certain regularity conditions on the symbol. For this case, we also mention works \cite{AkylzhanovRuzhansky,AkylzhanovRuzhanskyNursultanov,Edwards,NursultanovTleukhanova1999,NursultanovTleukhanova2000,PerssonSarybekovaTleukhanova,RozendaalVeraar,SarybekovaTararykovaTleukhanova} and references therein.

In this work we are interested on $L_p(I)\rightarrow L_q(I)$ boundedness of a Fourier multipliers in cases $I=\mathbb{R}$ and $I=(0,1)$. The corresponding higher dimensional cases will be considered in future work.

\subsection{H\"ormander type theorem.} 
We recall that in \cite[Theorem 1.11]{Hormander}, H\"ormander showed that, for $1 < p \leq 2 \leq q<\infty$, a symbol $\lambda$ and the corresponding Fourier transform multiplier $T_\lambda$ satisfy
\begin{equation}\label{H_t_cond}
	\|T_{\lambda}\|_{L_p(\mathbb{R})\mapsto L_q(\mathbb{R})} \lesssim \|\lambda\|_{L_{r,\infty}(\mathbb{R})}, \quad 1/r = 1/p -1/q.
\end{equation}
This result was also obtained for the case of interval. It was shown in \cite[p. 303]{Edwards} that under the same conditions on $p$, $q$, and $r$, for a sequence of complex numbers $\lambda = \{\lambda_k\}_{k\in \mathbb{Z}}$ and the corresponding Fourier series multiplier $T_\lambda$, the estimate holds
\begin{equation}\label{H_s_cond}
	\|T_{\lambda}\|_{L_p(0,1)\mapsto L_q(0,1) } \lesssim \|\lambda\|_{l_{r,\infty}(\mathbb{Z})}.
\end{equation}

There were some other works in this direction. In \cite{NursultanovTleukhanova1999,NursultanovTleukhanova2000}, authors improved the sufficient condition \eqref{H_s_cond}. Moreover, they obtain a necessary condition. We also mention that H\"ormander type theorem was obtained in \cite{AkylzhanovRuzhansky,AkylzhanovRuzhanskyNursultanov}, where author investigate he $L_p\rightarrow L_q$ boundedness of Fourier multipliers in the context of compact Lie groups.

This work partially devoted to further development of H\"ormander's result.  We weaken the sufficient conditions \eqref{H_t_cond} and \eqref{H_s_cond}. Additionally, we obtain necessary conditions for $L_p\rightarrow L_q$ boundedness of Fourier multipliers. In the interest of brevity, we present a slightly simplified version of our results:
\begin{theorem}\label{HT}
	Let $1< p \leq 2 \leq q <\infty$ and $1/r = 1/p - 1/q$. Let $r'$ be the conjugate exponent of $r$, then, the following statements are true
	\begin{enumerate}[label=(\roman*)]
		\item For a measurable function $\lambda$, it follows
		\begin{equation*}
			\sup_{k\in \mathbb{Z}} \sup_{e\in M_k} \frac{1}{|e|^{1/r'}} \left| \int_{e} \lambda(\xi) d\xi \right| \lesssim \|T_\lambda\|_{L_p(\mathbb{R})\rightarrow L_q(\mathbb{R})} \lesssim \sup_{k\in \mathbb{Z}} \sup_{e \subset \Delta_k} \frac{1}{|e|^{1/r'}} \left| \int_{e} \lambda(\xi) d\xi \right|,
		\end{equation*}
		where $M_k$ is the set of intervals containing in 
		$$\Delta_{k}: = (-2^{k+1}, -2^k] \cup [2^k,2^{k+1}).$$
		\item For a sequence $\lambda = \{\lambda_k\}_{k\in \mathbb{Z}}$, it follows
		\begin{equation*}
			\sup_{k\in \mathbb{N}_0} \sup_{e\in W_k} \frac{1}{|e|^{1/r'}} \left|  \sum_{m\in e} \lambda_m \right| \lesssim \|T_\lambda\|_{L_p(0,1)\rightarrow L_q(0,1)} \lesssim \sup_{k\in \mathbb{N}_0} \sup_{e \subset \delta_k} \frac{1}{|e|^{1/r'}} \left|  \sum_{m\in e} \lambda_m \right|,
		\end{equation*}
		where $\mathbb{N}_0 = \mathbb{N}\cup \{0\}$ and $W_k$ is the set of all discrete intervals (finite arithmetic progressions with a common difference of 1) containing in 
		\begin{equation*}
			\delta_k: = \{-2^{k+1}+1, \cdots, -2^k\}\cup \{2^k,\cdots,2^{k+1} - 1\},
		\end{equation*}
		If $k\in \mathbb{N}$, and $\delta_0:=\{-1,0,1\}$, if $k=0$.
	\end{enumerate}
\end{theorem}
A comparison of our results with those of previous studies is presented in Section \ref{examandcor}. It is shown that the sufficient conditions in Theorem \ref{HT} are strictly weaker than \eqref{H_t_cond} and \eqref{H_s_cond}.

For a large class of functions, we show that the sufficient and necessary conditions we obtained here are equivalent, allowing us to formulate a criteria for the $L_p\rightarrow L_q$ boundedness of Fourier multipliers. Namely, we say that 
a complex valued function $\lambda$ on $\mathbb{R}$ is a $M$-generalized monotone function if 
\begin{equation*}
	\lambda^*(t) \leq C \sup_{e\in M, |e|\geq t} \frac{1}{|e|} \left| \int_e \lambda(x) dx \right|,
\end{equation*}
where $M$ is a set of some measurable subsets in $\mathbb{R}$ with positive measure, $\lambda^*$ is a non-increasing rearrangement, and $C$ is some positive constant depending on $\lambda$. We show that if $M$ is the set of all intervals and $\lambda$ is a $M$-generalized monotone function, then $\lambda$ represents $L_p\rightarrow L_q$ Fourier transform multiplier if and only if 
\begin{equation*}
	\sup_{e \in M} \frac{1}{|e|^{1/r'}} \left| \int_{e} \lambda(\xi) d\xi \right| < \infty.
\end{equation*}
Analogically, we define $M$-generalized monotone sequences and obtain the same criteria for $L_p\rightarrow L_q$ boundedness of Fourier series multipliers. We mention that another generalizations of monotone functions and sequences were studied in \cite{Tikhonov2007,LiflyandTikhonov2008,LiflyandTikhonov2009,LiflyandTikhonov2011,DyachenkoMukanovTikhonov}.

\subsection{Lizorkin type theorem.}
Next, we recall another classical result. Under assumption $1<p<q<\infty$, Lizorkin \cite{Lizorkin} showed that for a continuously differentiable function $\lambda$ on $\mathbb{R}$, the corresponding Fourier transform multiplier $T_\lambda$ satisfies
\begin{equation}\label{L_t_cond}
	\|T_\lambda\|_{L_p(\mathbb{R})\mapsto L_q(\mathbb{R})} \lesssim \sup_{\xi\in \mathbb{R}} \left(  	|\xi|^{\frac{1}{r}}|\lambda(\xi)| +  |\xi|^{\frac{1}{r} + 1} |\lambda'(\xi)| \right),    \quad 1/r = 1/p -1/q.
\end{equation}
A similar result holds for intervals as well: For a sequence $\lambda=\{\lambda_k\}_{k\in \mathbb{Z}}$, the corresponding Fourier series multiplier $T_\lambda$ satisfies
\begin{equation}\label{L_s_cond}
		\|T_\lambda\|_{L_p(0,1)\mapsto L_q(0,1)} \lesssim \sup_{k \in \mathbb{Z}} \left(  |k|^{\frac{1}{r}} |\lambda_k|  +   |k|^{\frac{1}{r} + 1} |\lambda_k - \lambda_{k+1}|  \right).
\end{equation}
These results were generalized in \cite{SarybekovaTararykovaTleukhanova} and \cite{PerssonSarybekovaTleukhanova}. Authors derive strictly weaker sufficient conditions.

In this work we make further improvements of these results. We prove the following Lizorkin type theorem: 
\begin{theorem}\label{LT}
	Let $1 <p < q<\infty$, $1/r = 1/p - 1/q$, and $\Delta_{k}$, $\delta_{k}$ be the sets defined in Theorem \ref{HT}. Then the following statements are true
	\begin{enumerate}[label=(\roman*)]
		\item Let $\lambda$ be a real-valued function on $\mathbb{R}$ which is absolutely continuous on $(-\infty,0]$ and $[0,\infty)$ such that $\lambda(\xi)\rightarrow 0$ as $|\xi|\rightarrow \infty$. Then the corresponding Fourier transform multiplier satisfies
		\begin{equation*}
			\|T_\lambda\|_{L_p(\mathbb{R})\rightarrow L_q(\mathbb{R})} \lesssim \sup_{k\in \mathbb{Z}} 2^{\frac{k}{r}} \int_{\Delta_k} |\lambda'(\xi)| d\xi.
		\end{equation*}
		\item Let $\lambda=\{\lambda_k\}_{k\in \mathbb{Z}}$ be a sequence of real numbers such that $\lambda_k \rightarrow 0$ as $k\rightarrow \infty$. Then the corresponding Fourier series multiplier satisfies
	\begin{equation*}
		\|T_\lambda\|_{L_p(0,1)\rightarrow L_q(0,1)} \lesssim\sup_{k\in \mathbb{N}_0} 2^{\frac{k}{r}}\sum_{m=2^k}^{2^{k+1} - 1} \left(|\lambda_{ -m} - \lambda_{ -m + 1}| + |\lambda_{ m} - \lambda_{ m - 1}|\right) .
	\end{equation*}
	\end{enumerate}
\end{theorem}
We note that the sufficient conditions in Theorem \ref{LT} are strictly weaker than \eqref{L_t_cond} and \eqref{L_s_cond}, see Examples \ref{examL1} and \ref{examL2}. We show that Theorem \ref{LT} is at least complementary to results in \cite{SarybekovaTararykovaTleukhanova} and \cite{PerssonSarybekovaTleukhanova}.

The paper has simple structure. In Section \ref{prel}, we introduce notation and recall some definitions and know results. In Section \ref{sec_suf_cond}, we prove sufficient conditions, namely, we obtain H\"ormander type and Lizorkin type theorems. In Section \ref{sec_nec_cond}, we obtain necessary conditions. In the next section, we introduce notion of $M$-generalized monotone functions and sequences. We obtain criteria for boundedness of Fourier multipliers corresponding to $M$-generalized monotone functions and sequences. Finally, in Section \ref{examandcor}, we derive corollaries, give examples and compare our theorems with some previous results.

\section{Preliminaries}\label{prel}
In this section we introduce notations and recall some definitions. In our analysis, we often write $x\lesssim y$ or $y\gtrsim x$ to mean that $x\leq Cy$, where $C >0$ is some constant. The dependencies of $C$ will either be explicitly specified or otherwise, clear from context. By $x\approx y$ we mean that $x\lesssim y$ and $x\gtrsim y$. For two sequences $a=\{a_k\}_{k\in \mathbb{Z}}$ and $b=\{b_k\}_{k\in \mathbb{Z}}$, we write $ab:=\{a_kb_k\}_{k\in \mathbb{Z}}$. We also use notation $\mathbb{N}_0: = \mathbb{N}\cup \{0\}$, where $\mathbb{N}:= \{1,2,\cdots\}$.

Let $(\Omega,\Sigma, \mu)$ be a measure space. We denote by $L_p(\Omega)$, $1\leq p\leq \infty$, the space of measurable functions in $\Omega$ with integrable $p\textsuperscript{th}$ power, and write
\begin{equation*}
	\|f\|_{L_p(\Omega)}:= \left(\int_{\Omega} |f(x)|^pd_\mu x\right)^{1/p}, \qquad f\in L_p(\Omega).
\end{equation*}
When $p=\infty$ this is understood as the essential supremum of $|f|$. When $1\leq p\leq \infty$ we use notation $p'$ for the conjugate exponent defined by $1/p + 1/p' = 1$.

For a measurable function $f$, by $d_f$ and $f^*$ we denote its distribution function and non-increasing rearrangement:
\begin{equation*}
d_f(\sigma):=|\{x\in \Omega: |f(x)|\geq \sigma\}|, \qquad f^*(t):= \inf\{\sigma>0: d_f(\sigma)<t\}.
\end{equation*}

\begin{definition}
Let $0<p<\infty$ and $0<q\leq \infty$. The Lorentz space, $L_{p,q}(\Omega)$, is defined as the space of finitely measurable functions $f$ such that $\|f\|_{L_{p,q}(\Omega)} \leq \infty$, where
\begin{equation*}
	\|f\|_{L_{p,q}(\Omega)}:=
	\begin{cases}
	\left(\int_{0}^{\infty} \left(t^{\frac{1}{p}}f^*(t)\right)^q\frac{dt}{t}\right)^{\frac{1}{q}}, & \text{for } q<\infty,\\
	\sup_{t\geq 0}t^{\frac{1}{p}}f^*(t) <\infty, & \text{for } q=\infty.
	\end{cases}
\end{equation*}
\end{definition}
For $\Omega=\mathbb{Z}$ with $\Sigma = 2^{\mathbb{N}}$ and $\mu= \#$ being the power set of $\mathbb{Z}$ and counting measure, respectively, the non-increasing rearrangement, $a^*=\{a_k^*\}_{k\in \mathbb{N}}$, of $a=\{a_k\}_{k\in \mathbb{Z}}$ can be be obtained by permuting $\{|a_k|\}_{k\in \mathbb{Z}}$ in the non-increasing order. For this case, the definition becomes as follows:  
\begin{definition}
Let $0<p<\infty$ and $0<q\leq \infty$. The Lorentz sequence space, $l_{p,q}(\mathbb{Z})$, is defined as a space of sequences $a = \{a_k\}_{k\in \mathbb{Z}}$ such that $\|a\|_{l_{p,q}(\mathbb{Z})}<\infty$, where 
\begin{equation*}
\|a\|_{l_{p,q}(\mathbb{Z})}:=
\begin{cases}
\left(\sum_{k\in \mathbb{N}} \left(k^{\frac{1}{p}}a^*_k\right)^q\frac{1}{k}\right)^{\frac{1}{q}}, & \text{for } q<\infty,\\
\sup_{k\in \mathbb{N}}k^{\frac{1}{p}}a_k^* <\infty, & \text{for } q=\infty.
\end{cases}
\end{equation*}
Furthermore, for a subset $B\subset \mathbb{Z}$, we define
\begin{equation*}
	\|a\|_{l_{p,q}(B)} := \|\tilde{a}\|_{l_{p,q}(\mathbb{Z})},
\end{equation*}
where $\tilde{a}$ is a sequence such that 
\begin{equation*}
	\tilde{a}_k = 
	\begin{cases}
		a_k & \text{if } k\in B,\\
		0 & \text{if } k \in \mathbb{Z}\setminus B.
	\end{cases}
\end{equation*}

\end{definition}
\begin{remark}
	When $q=\infty$, the definitions above also make sense for $p=\infty$, so that the spaces $L_{\infty,\infty}$ and $l_{\infty,\infty}$ are well defined and they coincide with $L_{\infty}$ and $l_{\infty}$, respectively.
\end{remark}

Let $S(\mathbb{R})$ be the space of Schwartz functions on $\mathbb{R}$. For a function $\lambda$, the Fourier transform multiplier, $T_\lambda$, is given by the multiplication on the Fourier transform side, that is
\begin{equation*}
	\mathcal{F}T_\lambda f(\xi) = \lambda(\xi)\mathcal{F}f(\xi) \qquad \xi \in \mathbb{R}, \quad f\in S(\mathbb{R}),
\end{equation*}
where $\mathcal{F}$ is the Fourier transform:
\begin{equation*}
	\mathcal{F}f(\xi) := \frac{1}{\sqrt{2\pi}} \int_{\mathbb{R}} f(x) e^{-ix\xi}dx.
\end{equation*}

We also recall the definition of the Fourier series multipliers. We say that the sequence of complex numbers $\lambda = \{\lambda_k\}_{k\in \mathbb{Z}}$ represents a Fourier series multiplier, $T_{\lambda}$, from $L_p(0,1)$ to $L_q(0,1)$ if for any $f\in L_p(0,1)$ with 
$$f\sim\sum_{k\in\mathbb{Z}}a_ke^{2\pi i k x},$$
there exists $f_\lambda \in L_q(0,1)$ with 
$$f_\lambda\sim\sum_{k\in\mathbb{Z}}\lambda_k a_ke^{2\pi i k x}$$
and the operator $T_{\lambda}: f \mapsto f_\lambda$ is bounded from $L_p(0,1)$ to $L_q(0,1)$.

By $M_{p}^q$ we denote the normed space of Fourier transform multipliers with the norm given by 
\begin{equation*}
	\|\lambda\|_{M_p^q}: = \|T_\lambda\|_{L_p\mapsto L_q}.
\end{equation*}
Similarly, we denote the normed Fourier series multipliers by $m_p^q$. 
\subsection{E.Nursultanov's NET space}
Here we recall the NET space which was introduced by E.Nursultanov in \cite{NursultanovE1998DAN, NursultanovE1998}. 
\begin{definition}\label{NET space}
	Let $0<p<\infty$ and $0<q\leq \infty$. Let $(\Omega,\Sigma,\mu)$ be a measure space and $M$ be a family of some measurable sets in $\Omega$ with finite positive measures. Then, E.Nursultanov's space, $N_{p,q}(M) = N_{p,q}(\Omega,M)$, is defined as the space of integrable on each $e\in M$ functions $f$ such that $\|f\|_{N_{p,q}(M)}<\infty$, where
	\begin{equation*}
		\|f\|_{N_{p,q}(M)}:=
		\begin{cases}
			\left(\int_{0}^{\infty}\left(t^{\frac{1}{p}}\bar{f}(t,M)\right)^q \frac{dt}{t}\right)^{\frac{1}{q}}, & \text{for } q<\infty,\\
			\sup_{t> 0} t^{\frac{1}{p}}\bar{f}(t,M), & \text{for } q=\infty
		\end{cases}
	\end{equation*}
and\footnote{For $e\in\Sigma$ we write $|e|:= \mu(e)$. From the context, it will be clear if it is the measure or absolute value.}
\begin{equation*}
	\bar{f}(t,M) := \sup_{|e|\geq t, e\in M} \frac{1}{|e|} \left| \int_{e} f(x) d_\mu x \right|.
\end{equation*}
\end{definition}

For the sake of convenience, we repeat this definition for the case $(\Omega,\Sigma,\mu) = (\mathbb{Z}, 2^{\mathbb{N}}, \#)$.
\begin{definition}\label{Net seq space}
	Let $0<p<\infty$ and $0<q\leq \infty$. Let $W$ be a set of some finite non-empty subsets of $\mathbb{Z}$, then E.Nursultanov's sequence space, $n_{p,q}(W) = n_{p,q}(\mathbb{Z},W)$, is defined as a space of complex sequences $a = \{a_k\}_{k\in \mathbb{Z}}$ such that $\|a\|_{n_{p,q}(W)}<\infty$, where 
	\begin{equation*}
		\|a\|_{n_{p,q}(W)}:=
		\begin{cases}
			\left(\sum_{k\in \mathbb{N}}\left(k^{\frac{1}{p}}\bar{a}_k(W)\right)^q \frac{1}{k}\right)^{\frac{1}{q}}, & \text{for } q<\infty,\\
			\sup_{k\in \mathbb{N}} k^{\frac{1}{p}}\bar{a}_k(W), & \text{for } q=\infty,
		\end{cases}
	\end{equation*}
	and
	\begin{equation*}
		\bar{a}_k(W) := \sup_{|e|\geq k, e \in W} \frac{1}{|e|} \left|\sum_{j\in e}a_j\right|.
	\end{equation*}
\end{definition}

\begin{remark}
		Let us note that if $M$ is the set of all measurable subsets of $\Omega$ with finite positive measures, then $N_{p,q}(\Omega, M) = L_{p,q}(\Omega)$. Similarly, if $W$ be set of all finite non-empty subsets of $\mathbb{Z}$, then $n_{p,q}(\mathbb{Z},W) = l_{p,q}(\mathbb{Z})$.
\end{remark}
	Like Lorentz spaces, E.Nursultanov's spaces are nested increasingly with respect to the second parameter; see Remark 1 in \cite{NursultanovE2000} and Proposition 2 in \cite{NursultanovE1998}. More precisely:
	\begin{proposition}\label{inclusion}
		Let $0<p<\infty$ and $0<q_1\leq q_2\leq \infty$, then the following statements are true:
		\begin{enumerate}[label=(\roman*)]
			\item  Let $M$ be a family of some measurable subsets of $\Omega$ with finite positive measures.  Then $N_{p,q_1}(M) \hookrightarrow N_{p,q_2}(M)$, that is 
			\begin{equation*}
				\| f\|_{N_{p,q_2}(M)} \lesssim \| f\|_{N_{p,q_1}(M)} ,  \qquad   \text{for } f \in N_{p,q_1}(M).
			\end{equation*}
			\item Let $W$ be a family of some finite non-empty sets in $\mathbb{Z}$. Then $n_{p,q_1}(W) \hookrightarrow n_{p,q_2}(W)$, that is 
			\begin{equation*}
				\| a\|_{n_{p,q_2}(W)} \lesssim \| a\|_{n_{p,q_1}(W)}, \qquad \text{for } a \in n_{p,q_1}(W).
			\end{equation*}
		\end{enumerate}
	\end{proposition}

	Next, we give alternative expressions for the quasi-norms $\|\cdot\|_{N_{p,\infty}}$ and $\|\cdot\|_{n_{p,\infty}}$.
	
	\begin{proposition}\label{equiv norm}
		Let $(\Omega,\Sigma,\mu)$ be a measure space and $0< p<\infty$. Let $M\subset \Sigma$ be a fixed set, whose elements have finite positive measure. Then, for a function $f$ integrable over each $e\in M$, it follows
		\begin{equation*}
				\|f\|_{N_{p,\infty}(M)} = \sup_{e\in M} \frac{1}{|e|^{\frac{1}{p'}}} \left|\int_e f(x)d_\mu x\right|.
		\end{equation*}
	
		In particular, if $(\Omega,\Sigma,\mu) = (\mathbb{Z}, 2^{\mathbb{N}}, \#)$ and $W$ being some fixed set of finite non-empty  subsets of $\mathbb{Z}$, then
		\begin{equation*}
			\|a\|_{n_{p,\infty}(W)} = \sup_{e\in W}  \frac{1}{|e|^{\frac{1}{p'}}} \left|\sum_{k\in e} a_k\right|.
		\end{equation*}
		for any sequence of complex numbers $a = \{a_k\}_{k\in \mathbb{Z}}$.
	\end{proposition}
	\begin{proof}
		For any $e_0 \in M$, we estimate
		\begin{multline*}
			\|f\|_{N_{p,\infty}(M)} = \sup_{t> 0} t^{\frac{1}{p}}\sup_{|e|\geq t, e\in M} \frac{1}{|e|} \left| \int_{e} f(x) d_\mu x \right| \geq |e_0|^{\frac{1}{p}}\sup_{|e|\geq |e_0|, e\in M} \frac{1}{|e|} \left| \int_{e} f(x) d_\mu x \right|\\
			 \geq \frac{1}{|e_0|^{1/p'}} \left| \int_{e_0} f(x) d_\mu x \right|.
		\end{multline*}
	Conversely, 
	\begin{multline*}
		\|f\|_{N_{p,\infty}(M)} = \sup_{t> 0} t^{\frac{1}{p}}\sup_{|e|\geq t, e\in M} \frac{1}{|e|} \left| \int_{e} f(x) d_\mu x \right| \sup_{t> 0} \leq \sup_{t> 0}\sup_{|e|\geq t, e\in M}|e|^{\frac{1}{p}} \frac{1}{|e|} \left| \int_{e} f(x) d_\mu x \right| \\
		= \sup_{e\in M}|e|^{\frac{1}{p}} \frac{1}{|e|} \left| \int_{e} f(x) d_\mu x \right| .
	\end{multline*}
	\end{proof}
	
	From now on, we only consider the cases where $\Omega$ is $\mathbb{R}$ or $\mathbb{Z}$ and $\mu$ being the Lebesgue or counting measure, respectively.
	\begin{lemma}\label{equiv p,1}
		Let $0<p<\infty$ and $0<q\leq \infty$, then the following statements are true:
		\begin{enumerate}[label=(\roman*)]
			\item Let $M$ be a family of some measurable sets in $\mathbb{R}$ with finite positive measures. Then we have the equivalence
			\begin{equation*}
				\|f\|_{N_{p,q}(M)} \approx 
				\begin{cases}
					\left(\sum_{k\in \mathbb{Z}} \left(2^{\frac{k}{p}} \bar{f}(2^k,M)\right)^q\right)^{1/q}, & q\neq \infty,\\
					\sup_{k\in \mathbb{Z}} 2^{\frac{k}{p}} \bar{f}(2^k,M), & q=\infty.
				\end{cases}
			\end{equation*}
			\item Let $W$ be a family of some finite non-empty sets in $\mathbb{Z}$. Then
			\begin{equation*}
				\|a\|_{n_{p,q}(W)} \approx
				\begin{cases}
						\left( \sum_{n\in \mathbb{N}_0} \left(2^{\frac{n}{p}}\bar{a}_{2^n}(W)\right)^q\right)^{1/q}, & q\neq \infty,\\
						\sup_{k\in \mathbb{N}_0} (2^{\frac{k}{p}}\bar{a}_{2^k}(W), & q=\infty.
				\end{cases}
			\end{equation*}
		\end{enumerate}
	\end{lemma}
	\begin{proof}
		By definition,
		\begin{multline*}
			\|f\|_{N_{p,q}(M)}^q  = \int_{0}^{\infty} \left(t^{1/p} \bar{f}(t,M)\right)^q\frac{dt}{t} = \sum_{k\in \mathbb{Z}} \int_{2^k}^{2^{k+1}} \left(t^{1/p}\bar{f}(t,M)\right)^q \frac{dt}{t}\\
			\approx  \sum_{k\in \mathbb{Z}} \left(2^{\frac{k}{p}} \bar{f}(2^k,M)\right)^q.
		\end{multline*}
		and
		\begin{multline*}
			\|a\|_{n_{p,q}(W)}^q = \sum_{k\in \mathbb{N}} \left(k^{1/p} \bar{a}_k(W)\right)^q \frac{1}{k} = \sum_{n\in \mathbb{N}_0}  \sum_{k=2^n}^{2^{n+1}-1}  \left(k^{1/p} \bar{a}_k(W)\right)^q \frac{1}{k} \approx \sum_{n\in \mathbb{N}_0} \left(2^{\frac{n}{p}}\bar{a}_{2^n}(W)\right)^q.
		\end{multline*}
	Similarly, one can obtain the corresponding formulas for the case $q = \infty$.
	\end{proof}
	Finally,  we will state known results which will be used later in this work. The first part of the following theorem was proved in \cite{NursultanovE1998DAN} and the second part in \cite[Theorem 3]{NursultanovE1998}
	\begin{theorem}\label{nursultanovs inequality}
		Let $2\leq p<\infty$ and $0<q\leq \infty$, then the following statements are true:
		\begin{enumerate}[label=(\roman*)]
			\item Let $M$ be the set of all finite intervals on $\mathbb{R}$. Let $f\in L_{p,q}(\mathbb{R})$, then
			\begin{equation*}
			\|\mathcal{F}f\|_{N_{p',q}(M)}\lesssim \|f\|_{L_{p,q}(\mathbb{R})}.
			\end{equation*}
			\item Let $W$ be the set of all finite intervals on $\mathbb{Z}$. Let $f\in L_{p,q}(0,1)$ and $f\sim\sum_{k\in\mathbb{Z}}a_ke^{2\pi i k x}$, then
			\begin{equation*}
			\|a\|_{n_{p',q}(W)}\lesssim \|f\|_{L_{p,q}(0,1)}.
			\end{equation*}
		\end{enumerate}
	\end{theorem}
	
	\begin{remark}
		Originally, the above theorem was stated for $2<p<\infty$. However, by careful checking the proof, one can verify that Theorem \ref{nursultanovs inequality} holds also for $1<p<\infty$.
	\end{remark}

\section{Sufficient conditions}\label{sec_suf_cond}
In this section we obtain necessary conditions for the $L_p-L_q$ boundedness of Fourier multipliers, which imply the upper bounds in Theorems \ref{HT} and \ref{LT}.
\subsection{H\"ormander type theorems.}\label{Hsection}
We begin by proving H\"ormander type theorem for Fourier transform multipliers:
	\begin{theorem}\label{sufHtr}
		Let $1< p \leq 2 \leq q <\infty$ and $1/r = 1/p - 1/q$. Assume that $\lambda$ is a measurable function, then 
		\begin{equation*}
			\|\lambda\|_{M_p^q} \lesssim \sup_{k\in \mathbb{Z}} \|\lambda\|_{L_{r,\infty}(\Delta_{k})},
		\end{equation*}
	where
	$$\Delta_{k}: = (-2^{k+1}, -2^k] \cup [2^k,2^{k+1}).$$
	\end{theorem}
	
	\begin{proof}
		Let $f\in S(\mathbb{R})$, then 
		\begin{equation*}
		\|T_{\lambda}f\|_{L_q} = \|\mathcal{F}^{-1}\lambda\mathcal{F}f\|_{L_q}= \left\| \sum_{k\in\mathbb{Z}} \int_{\Delta_{k}} e^{i\xi\cdot x}\lambda(\xi) \mathcal{F}f(\xi) d\xi \right\|_{L_q} =  \left\|\sum_{k\in\mathbb{Z}} \mathcal{F}^{-1}\lambda\chi_{\Delta_{k}}\mathcal{F}f\right\|_{L_q},
		\end{equation*}
		where $\chi_{\Delta_{k}}$ is the indicator function of $\Delta_{k}$, that is
		\begin{equation*}
			\chi_{\Delta_{k}}(\xi):=
			\begin{cases}
				1 & \text{if } \xi\in \Delta_{k},\\
				0 & \text{otherwise}.
			\end{cases}
		\end{equation*}
		By using the Littlewood-Paley inequality, see \cite{Stein}, we write
		\begin{equation*}
		\|T_{\lambda}f\|_{L_q} \lesssim \left\|\left(\sum_{k\in\mathbb{Z}} \left|\mathcal{F}^{-1}\lambda\chi_{\Delta_{k}}\mathcal{F}f\right|^2\right)^{\frac{1}{2}}\right\|_{L_q}.
		\end{equation*}
		Since $q\geq2$ the Minkowski inequality gives
		\begin{equation*}
		\|T_{\lambda}f\|_{L_q} \lesssim \left(\sum_{k\in\mathbb{Z}} \left\|\mathcal{F}^{-1}\lambda\chi_{\Delta_{k}}\mathcal{F}f\right\|_{L_q}^2\right)^{\frac{1}{2}}.
		\end{equation*}
		Further, the Hardy-Littlewood inequality (if $q>2$) or the Parseval identity (if $q=2$) gives 
		\begin{equation*}
		\|T_{\lambda}f\|_{L_q} \lesssim \left(\sum_{k\in\mathbb{Z}} \left\|\lambda\chi_{\Delta_{k}}\mathcal{F}f\right\|_{L_{q',q}}^2\right)^{\frac{1}{2}}.
		\end{equation*}
		If $r = \infty$, that is $p = q= 2$, we estimate
		\begin{multline*}
			\|T_{\lambda}f\|_{L_q}  \lesssim \sup_{k\in\mathbb{Z}}\|\lambda\|_{L_{\infty}(\Delta_k)}^2 \left(\sum_{k\in\mathbb{Z}} \left\|\chi_{\Delta_{k}}\mathcal{F}f\right\|_{L_{2,2}(\Delta_k)}^2\right)^{\frac{1}{2}} \\
			\lesssim \sup_{k\in \mathbb{Z}} \|\lambda\|_{L_{r,\infty}(\Delta_{k})} \left(\sum_{k\in\mathbb{Z}} \left\|\chi_{\Delta_{k}}\mathcal{F}f\right\|_{L_{p',q}(\Delta_k)}^2\right)^{\frac{1}{2}}.
		\end{multline*}
		Otherwise, when $r<\infty$, we use the H\"older inequality to derive the same estimate
		\begin{multline*}
		\|T_{\lambda}f\|_{L_q}  \lesssim \left(\sum_{k\in\mathbb{Z}}\|\lambda\|_{L_{r,\infty}(\Delta_k)}^2 \left\|\chi_{\Delta_{k}}\mathcal{F}f\right\|_{L_{p',q}(\Delta_k)}^2\right)^{\frac{1}{2}} \\
		\lesssim \sup_{k\in \mathbb{Z}} \|\lambda\|_{L_{r,\infty}(\Delta_{k})} \left(\sum_{k\in\mathbb{Z}} \left\|\chi_{\Delta_{k}}\mathcal{F}f\right\|_{L_{p',q}(\Delta_k)}^2\right)^{\frac{1}{2}}.
		\end{multline*}
		Since $p\leq q$, it follows that $L_{p',p}(\Delta_{k})\hookrightarrow L_{p',q}(\Delta_{k})$ and
		\begin{equation*}
			\|f\|_{L_{p',q}(\Delta_{k})} \leq C\|f\|_{L_{p',p}(\Delta_{k})},  \qquad \text{for } f \in L_{p',p}(\Delta_{k}),
		\end{equation*}
		where $C>0$ is independent on $k\in \mathbb{Z}$. Therefore, the penultimate inequality gives 	
		\begin{equation*}
		\|T_{\lambda}f\|_{L_q} 
		\lesssim \sup_{k\in \mathbb{Z}} \|\lambda\|_{L_{r,\infty}(\Delta_{k})} \left(\sum_{k\in\mathbb{Z}} \left\|\chi_{\Delta_{k}}\mathcal{F}f\right\|_{L_{p',p}(\Delta_k)}^2\right)^{\frac{1}{2}}.
		\end{equation*}
		We will repeat our steps in reverse order given that $p\leq 2$. The Hardy-Littlewood inequality (or Parseval identity if $p=2$) gives 
		\begin{equation*}
		\|T_{\lambda}f\|_{L_q} \lesssim  \sup_{k\in \mathbb{Z}} \|\lambda\|_{L_{r,\infty}(\Delta_{k})} \left(\sum_{k\in\mathbb{Z}} \left\|\mathcal{F}^{-1}\chi_{\Delta_{k}}\mathcal{F}f\right\|_{L_p}^2\right)^{\frac{1}{2}}.
		\end{equation*}
		Since $p\leq 2$, by Minkowski inequality, we obtain
		\begin{equation*}
		\|T_{\lambda}f\|_{L_q} \lesssim  \sup_{k\in \mathbb{Z}} \|\lambda\|_{L_{r,\infty}(\Delta_{k})}\left\|\left(\sum_{k\in\mathbb{Z}} \left|\mathcal{F}^{-1}\chi_{\Delta_{k}}\mathcal{F}f\right|^2\right)^{\frac{1}{2}}\right\|_{L_p}.
		\end{equation*}
		Finally, the Littlewood-Paley inequality implies that 
		\begin{equation*}
			\|T_{\lambda}f\|_{L_q} \lesssim \sup_{k\in \mathbb{Z}} \|\lambda\|_{L_{r,\infty}(\Delta_{k})} \|f\|_{L_p}.
		\end{equation*}
	\end{proof}
	Next, we obtain analogue of this theorem but for the Fourier series multipliers:
	\begin{theorem}\label{Suf condition 1(i)}
		Let $1< p \leq 2 \leq q <\infty$ and $1/r = 1/p - 1/q$. Let $\lambda = \{\lambda_k\}_{k\in \mathbb{Z}}$ be a sequence of complex numbers, then
		\begin{equation*}
			\|\lambda\|_{m_p^q} \lesssim \sup_{k\in \mathbb{N}_0} \|\lambda\|_{l_{r,\infty}(\delta_{k})},
		\end{equation*}
	where
	\begin{equation*}
		\delta_k:= 
		\begin{cases}
			 \{-2^{k+1}+1, \cdots, -2^k\}\cup \{2^k,\cdots,2^{k+1} - 1\}, & k\in \mathbb{N},\\
			 \{-1,0,1\}, & k=0.
		\end{cases}
	\end{equation*}
	\end{theorem}
	\begin{proof}
		Let $f\in L_p(0,1)$ and $a = \{a_k\}_{k\in \mathbb{Z}}$ be its Fourier coefficients. By using the Littlewood-Paley inequality, we write
		\begin{equation*}
			\|T_\lambda f\|_{L_q(0,1)} \lesssim \left\| \left( \sum_{k\in \mathbb{N}_0} \left|\sum_{m\in \delta_k} \lambda_m a_me^{2\pi imx}\right|^2 \right)^{1/2}\right\|_{L_{q}(0,1)}.
		\end{equation*}
		Since $q\geq 2$, by the Minkowski inequality, we obtain
		\begin{equation*}
			\|T_\lambda f\|_{L_q(0,1)} \lesssim  \left( \sum_{k\in \mathbb{N}_0} \left\|\sum_{m \in \delta_{k}}\lambda_m a_me^{2\pi imx}\right\|_{L_{q}(0,1)}^2 \right)^{1/2}.
		\end{equation*}
		Further, the Hardy-Littlewood inequality (if $q>2$) or the Parseval identity (if $q=2$) gives 
		\begin{equation*}
			\|T_\lambda f\|_{L_q(0,1)} \lesssim \left( \sum_{k\in \mathbb{N}_0} \|\lambda a\|_{l_{q',q}\left(\delta_{k}\right)}^2 \right)^{1/2}.
		\end{equation*}
		If $r = \infty$, that is $p = q =2$, we estimate
		\begin{multline*}
			\|T_\lambda f\|_{L_q(0,1)} \lesssim \sup_{k\in \mathbb{N}_0}\|\lambda\|_{l_{\infty}\left(\delta_{k}\right)}\left( \sum_{k\in \mathbb{N}_0} \| a\|_{l_{2,2}\left(\delta_{k}\right)}^2 \right)^{1/2}\\
			\lesssim \sup_{k\in \mathbb{N}_0} \|\lambda\|_{l_{r,\infty}(\delta_{k})} \left(\sum_{k\in \mathbb{N}_0} \|a\|_{l_{p',q}(\delta_{k})}^2 \right)^{1/2}.
		\end{multline*}
		Otherwise, when $r<\infty$, we use the H\"older inequality to derive the same estimate
		\begin{multline*}
			\|T_\lambda f\|_{L_q(0,1)} \lesssim \left(\sum_{k\in \mathbb{N}_0} \|\lambda\|_{l_{r,\infty}\left(\delta_{k}\right)}^2\|a\|_{l_{p',q}\left(\delta_{k}\right)}^2\right)^{1/2}\\
			\lesssim \sup_{k\in \mathbb{N}_0} \|\lambda\|_{l_{r,\infty}(\delta_{k})} \left(\sum_{k\in \mathbb{N}_0} \|a\|_{l_{p',q}(\delta_{k})}^2 \right)^{1/2}
		\end{multline*}
		Since $p\leq q$, we know that $l_{p',p} \hookrightarrow l_{p',q}$ and the corresponding inequality does not depend on $k\in \mathbb{N}_0$. Therefore, the last estimate gives
		\begin{equation*}
			\|T_\lambda f\|_{L_q(0,1)}
			\lesssim \sup_{k\in \mathbb{N}_0} \|\lambda\|_{l_{r,\infty}(\delta_{k})}\left(\sum_{k\in \mathbb{N}_0} \|a\|_{l_{p',p}(\delta_{k})}^2 \right)^{1/2}.
		\end{equation*}
		We will repeat our steps in reverse order given that $p\leq 2$. The Hardy-Littlewood inequality (or Parseval identity if $p=2$) gives 
		\begin{equation*}
			\|T_\lambda f\|_{L_q(0,1)} 
			\lesssim \sup_{k\in \mathbb{N}_0} \|\lambda\|_{l_{r,\infty}(\delta_{k})}  \left( \sum_{k\in \mathbb{N}_0} \left\|\sum_{m\in \delta_{k}} a_me^{2\pi imx}\right\|_{L_{p,p}(0,1)}^2  \right)^{1/2}.
		\end{equation*}
		By the Minkowski inequality,
		\begin{equation*}
			\|T_\lambda f\|_{L_q} 
			\lesssim  \sup_{k\in \mathbb{N}_0} \|\lambda\|_{l_{r,\infty}(\delta_{k})} \left\| \left( \sum_{k\in \mathbb{N}_0} \left(\sum_{m\in \delta_{k} } a_me^{2\pi imx}\right)^2 \right)^{1/2}\right\|_{L_{p}(0,1)}.
		\end{equation*}
		Finally, Littlewood-Paley inequality implies
		\begin{equation*}
			\|T_\lambda f\|_{L_q(0,1)}\lesssim \sup_{k\in \mathbb{N}_0} \|\lambda\|_{l_{r,\infty}(\delta_{k})} \|f\|_{L_p(0,1)}.
		\end{equation*}
	\end{proof}
	\subsection{Lizorkin type theorems} Here, we obtain Lizorkin type theorems. We start with the Fourier transform multipliers:
	\begin{theorem}
		Let $1<p<q<\infty$ and $\lambda$ be a real-valued function on $\mathbb{R}$ which is absolutely continuous on $(-\infty,0]$ and $[0,\infty)$ such that 
		\begin{equation}\label{LT1}
			\lambda(\xi) \rightarrow 0 \quad as \quad |\xi| \rightarrow \infty,
		\end{equation}
		\begin{equation}\label{LT2}
			\sup_{k\in \mathbb{Z}} 2^{k(\frac{1}{p} - \frac{1}{q})} \int_{\Delta_k} |\lambda'(\xi)| d\xi < A <\infty,
		\end{equation}
		for some constant $A>0$ and
		$$\Delta_{k}: = (-2^{k+1}, -2^k] \cup [2^k,2^{k+1}).$$
		Then $\lambda \in M_p^q$ and
		\begin{equation*}
			\|\lambda\|_{M_p^q} \lesssim A.
		\end{equation*}
	\end{theorem}
	\begin{proof}
		First, we prove that $T_\lambda$ is a bounded operator from $L_{p,1}(\mathbb{R})$ to $L_{q,\infty}(\mathbb{R})$. We estimate
		\begin{multline*}
			\|T_\lambda\|_{L_{p,1} \mapsto L_{q,\infty}} = \sup_{\|f\|_{L_{p,1}} = 1} \|T_{\lambda}f\|_{L_{q,\infty}} \lesssim \sup_{\|f\|_{L_{p,1}} = 1} \|T_{\lambda}f\|_{L_{q}}\\
			 = \sup_{\|f\|_{L_{p,1}} =\|g\|_{L_{q'}} = 1} \int_{\mathbb{R}} T_{\lambda}f(x)g(x) dx.
		\end{multline*}
		Then, by the Parseval's identity, we obtain
		\begin{align}\label{p1qinfty est}
			\|T_\lambda\|_{L_{p,1} \mapsto L_{q,\infty}} \lesssim \sup_{\|f\|_{L_{p,1}} =\|g\|_{L_{q'}} = 1} \int_{\mathbb{R}} \lambda(\xi) \mathcal{F}f(\xi) \mathcal{F}g(\xi) d\xi.
		\end{align}
		Let us denote
		\begin{equation*}
			\phi(\xi) := \int_{0}^{\xi}\mathcal{F}f(\zeta) \mathcal{F}g(\zeta)d\zeta, \qquad I_1:= \int_{0}^{\infty} \lambda(\xi) \phi'(\xi) d\xi, \qquad I_2:= \int_{-\infty}^0 \lambda(\xi) \phi'(\xi) d\xi.
		\end{equation*}
		By integration by parts, we obtain
		\begin{equation*}
			|I_1| = \left|\int_{0}^{\infty} \lambda(\xi) \phi'(\xi)d\xi\right| = \left|\int_{0}^{\infty} \lambda'(\xi) \phi(\xi)d\xi\right| = \left|\sum_{k\in \mathbb{Z}} \int_{\Delta_{k}^{+}} \lambda'(\xi)2^{k\left(\frac{1}{p} - \frac{1}{q} \right)} 2^{k\left(\frac{1}{q} - \frac{1}{p} \right)}\phi(\xi)d\xi\right| ,
		\end{equation*}
		where $\Delta_{k}^{+}:=[2^k,2^{k+1})$. Using the hypothesis of the theorem, we conclude that
		\begin{multline*}
			|I_1|\leq A \sum_{k\in \mathbb{Z}} 2^{k\left(\frac{1}{q} - \frac{1}{p} + 1\right)} \sup_{\xi\in \Delta_{k}^+}\frac{1}{2^k} |\phi(\xi)| \lesssim  A \sum_{k\in \mathbb{Z}} 2^{k\left(\frac{1}{q} + \frac{1}{p'}\right)} \sup_{\xi \in\Delta_k^+}\frac{1}{\xi} \left|\int_{0}^\xi \mathcal{F}f(\zeta) \mathcal{F}g(\zeta)d\zeta\right|\\
			\lesssim A\sum_{k\in \mathbb{Z}} 2^{k\left(\frac{1}{q} + \frac{1}{p'}\right)} \sup_{e\in M, |e|\geq 2^k}\frac{1}{|e|} \left|\int_{e} \mathcal{F}f(\zeta) \mathcal{F}g(\zeta)d\zeta\right|.
		\end{multline*}
		Let $r>0$ be such that 
		$1/r = 1/q + 1/p'$, then Lemma \ref{equiv p,1} implies
		\begin{equation*}
			|I_1| \lesssim A\|\mathcal{F}f\mathcal{F}g\|_{N_{r,1}(M)} = A\|\mathcal{F}(f*g)\|_{N_{r,1}(M)}.
		\end{equation*}
		By Theorem \ref{nursultanovs inequality} and O'Neil inequality, we obtain 
		$$|I_1| \lesssim A \|f*g\|_{L_{r',1}} \lesssim A\|f\|_{L_{p,1}}\|g\|_{L_{q',\infty}}.$$
		Similarly, one can derive the same upper-bound for $|I_2|$. Putting these inequalities into \eqref{p1qinfty est} gives
		\begin{equation}\label{p1 qinfty series}
			\|T_\lambda\|_{L_{p,1} \mapsto L_{q,\infty}} \lesssim A.
		\end{equation}
		Let us pic $p_0, p_1$ such that $1<p_0 < p <p_1<\infty$ and choose $q_0$, $q_1$ so that
		\begin{equation}\label{pq0 pq1 r series}
			\frac{1}{p_0} - \frac{1}{q_0} = \frac{1}{p_1} - \frac{1}{q_1} = \frac{1}{p} - \frac{1}{q}.
		\end{equation}
		Then, by \eqref{p1 qinfty series}, we know that 
		\begin{equation*}
			\|T_\lambda\|_{L_{p_j,1} \mapsto L_{q_j,\infty}} \lesssim A \qquad \text{for } j=0,1.
		\end{equation*}
		Since $p_0 < p <p_1$, there exists $0<\theta<1$ such that
		\begin{equation*}
			\frac{1}{p} = \frac{1-\theta}{p_0} + \frac{\theta}{p_1},
		\end{equation*}
		and hence, the relation \eqref{pq0 pq1 r series} gives
		\begin{equation*}
			\frac{1}{q} = \frac{1-\theta}{q_0} + \frac{\theta}{q_1},
		\end{equation*}
		Therefore, from the Marcinkiewicz-Calderon's interpolation theorem, it follows that $$\|\lambda\|_{M_p^q} = \|T_\lambda\|_{L_p\mapsto L_q} \lesssim A.$$
	\end{proof}
	Analogue of this result holds for Fourier series multipliers:
	\begin{theorem}
		Let $1<p<q<\infty$ and $1/r=1/p - 1/q$. Let $\lambda = \{\lambda_k\}_{k\in \mathbb{Z}}$ be a sequence such that 
		\begin{equation*}
			\lambda_k \rightarrow 0 \quad as \quad k\rightarrow\infty
		\end{equation*}
		and
		\begin{equation*}
			\sup_{k\in \mathbb{N}_0} 2^{\frac{k}{r}}\sum_{m=2^k}^{2^{k+1} - 1} \left(|\lambda_{ -m} - \lambda_{ -m + 1}| + |\lambda_{m} - \lambda_{ m - 1}|\right) \leq A,
		\end{equation*}
		for some constant $A>0$. Then $\lambda \in m_p^q$ and $\|\lambda\|_{m_p^q}\lesssim A$.
	\end{theorem}
	
	\begin{proof}
		First, we prove that $T_\lambda$ is a bounded operator from $L_{p,1}(0,1)$ to $L_{q,\infty}(0,1)$. We estimate
		\begin{align*}
			\|T_\lambda\|_{L_{p,1} \mapsto L_{q,\infty}} = \sup_{\|f\|_{L_{p,1}}=1} \|T_\lambda f\|_{L_{q,\infty}} \leq \sup_{\|f\|_{L_{p,1}}=1} \|T_{\lambda}f\|_{L_{q}} \leq \sup_{\|f\|_{L_{p,1}}=\|g\|_{L_{q'}}=1} \int_{0}^{1} T_{\lambda}f(x)g(x)dx.
		\end{align*}
		Therefore
		\begin{equation*}
			\|T_\lambda\|_{L_{p,1} \mapsto L_{q,\infty}} \leq \sup_{\|f\|_{L_{p,1}}=\|g\|_{L_{q'}}=1} \left|\sum_{m\in \mathbb{Z}} \lambda_m a_m b_m\right|,
		\end{equation*}
		where $\{a_k\}$ and $\{b_k\}$ are Fourier coefficients of functions $f$ and $g$, respectively. Since $\lambda_k\rightarrow 0$ as $k\rightarrow\infty$, by using Abel transform, we derive
		\begin{multline*}
			\|T_\lambda\|_{L_{p,1} \mapsto L_{q,\infty}} \\
			\lesssim \sup_{\|f\|_{L_{p,1}}=\|g\|_{L_{q'}}=1} \left( \sum_{m=1}^{\infty} |\lambda_m - \lambda_{m-1}| \left|\sum_{l=0}^{m - 1}a_lb_l\right| + \sum_{m=1}^{\infty} |\lambda_{-m} - \lambda_{-m+1}| \left|\sum_{l=0}^{m - 1}a_{-l}b_{-l}\right| \right)
		\end{multline*}
		Then, we estimate
		\begin{multline*}
			\|T_\lambda\|_{L_{p,1} \mapsto L_{q,\infty}} \\
			 \leq  \sup_{\|f\|_{L_{p,1}}=\|g\|_{L_{q'}}=1} \sum_{k=0}^{\infty} \sup_{e\in W, 2^k\leq|e|<2^{k+1}}\left|\sum_{l\in e}a_lb_l\right| \sum_{m=2^k}^{2^{k+1} - 1}\left( |\lambda_m - \lambda_{m-1}| + |\lambda_{-m} - \lambda_{-m+1}|\right)
		\end{multline*}
		Using the theorem's conditions, we obtain
		\begin{align*}
			\|T_\lambda\|_{L_{p,1} \mapsto L_{q,\infty}} &\leq  A \sup_{\|f\|_{L_{p,1}}=\|g\|_{L_{q'}}=1} \sum_{k=0}^{\infty} (2^k)^{1 - \frac{1}{r}} \sup_{e\in W, 2^k\leq|e|<2^{k+1}}\frac{1}{2^k}\left|\sum_{l\in e}a_lb_l\right| \\
			&\lesssim  A \sup_{\|f\|_{L_{p,1}}=\|g\|_{L_{q'}}=1} \sum_{k=0}^{\infty} (2^k)^{1 - \frac{1}{r}} \sup_{e\in W, |e|\geq 2^k}\frac{1}{|e|}\left|\sum_{l\in e}a_lb_l\right| .
		\end{align*}
		Let $\tau >0$ be a number such that $1/\tau = 1 - 1/r$. Using Lemma \ref{equiv p,1}, we derive 
		\begin{equation*}
				\|T_\lambda\|_{L_{p,1} \mapsto L_{q,\infty}} \lesssim A \sup_{\|f\|_{L_{p,1}}=\|g\|_{L_{q'}}=1} \|ab\|_{n_{\tau, 1}}.
		\end{equation*}
		From Theorem \ref{nursultanovs inequality}, it follows
		\begin{equation*}
			\|T_\lambda\|_{L_{p,1} \mapsto L_{q,\infty}} \lesssim A \sup_{\|f\|_{L_{p,1}}=\|g\|_{L_{q'}}=1} \| f*g\|_{L_{\tau',1}}.
		\end{equation*}
		Since
		\begin{equation*}
			1 +\frac{1}{\tau'} = 1 + 1 -\frac{1}{\tau} = 1 +\frac{1}{r} = \frac{1}{p} + \frac{1}{q'},
		\end{equation*}
		the O'Neil inequality gives
		\begin{equation}\label{p1 qinfty}
				\|T_\lambda\|_{L_{p,1} \mapsto L_{q,\infty}} \lesssim A\sup_{\|f\|_{L_{p,1}}=\|g\|_{L_{q'}}=1} \|f\|_{L_{p,1}} \|g\|_{L_{q',\infty}}\lesssim A.
		\end{equation}
		
		Let us pick $p_0, p_1$ such that $1<p_0 < p <p_1<\infty$ and choose $q_0$, $q_1$ so that
		\begin{equation}\label{pq0 pq1 r}
			\frac{1}{p_0} - \frac{1}{q_0} = \frac{1}{p_1} - \frac{1}{q_1} = \frac{1}{p} - \frac{1}{q}.
		\end{equation}
		Then, by \eqref{p1 qinfty}, we know that 
		\begin{equation*}
			\|T_\lambda\|_{L_{p_j,1} \mapsto L_{q_j,\infty}} \lesssim A, \qquad \text{for } j=0,1.
		\end{equation*}
		Since $p_0 <p_1$, there exists $0<\theta<1$ such that
		\begin{equation*}
			\frac{1}{p} = \frac{1-\theta}{p_0} + \frac{\theta}{p_1},
		\end{equation*}
		and hence, the relation \eqref{pq0 pq1 r} gives
		\begin{equation*}
			\frac{1}{q} = \frac{1-\theta}{q_0} + \frac{\theta}{q_1}.
		\end{equation*}
		Therefore, from the Marcinkiewicz-Calderon's interpolation theorem, it follows that $$\|\lambda\|_{m_p^q} = \|T_\lambda\|_{L_p\mapsto L_q} \lesssim A.$$ 
	\end{proof}
		
	\section{Necessary conditions}\label{sec_nec_cond}
	In this section, we derive sufficient condition for $L_p-L_q$ boundedness for Fourier multipliers. First, we obtain this for Fourier transform multipliers:
		\begin{theorem}\label{Nes condition transform}
			Let $1< p \leq 2 \leq q <\infty$ and $1/r = 1/p - 1/q$. Let $0<\tau\leq \infty$ and $M$ be the set of all finite intervals in $\mathbb{R}$. Then, for a measurable function $\lambda$, it follows
			\begin{equation*}
			\sup_{e\in M} \frac{1}{|e|^{1/r'}} \left| \int_{e} \lambda(\xi) d\xi \right| \lesssim \|T_\lambda\|_{L_p\mapsto L_{q,\tau}}.
			\end{equation*}
		\end{theorem}
		\begin{proof}
		Let $e_0$ be an arbitrary interval, that is $e_0\in M$. We choose $f$ such that $\mathcal{F}f = \chi_{e_0}$, where $\chi_{e_0}$ is the indicator function of $e_0$. By Theorem \ref{nursultanovs inequality}, we obtain 
		\begin{equation}\label{nach lb}
		\|T_{\lambda}f\|_{L_{q,\tau}(\mathbb{R})} = \|\mathcal{F}^{-1}\lambda\mathcal{F}
		f\|_{L_{q,\tau}(\mathbb{R})}\gtrsim \|\lambda\mathcal{F}
		f\|_{N_{q',\tau}(M)} \gtrsim \|\lambda\mathcal{F}
		f\|_{N_{q',\infty}(M)},
		\end{equation}
		where $N_{p,q}(M) = N_{p,q}(\mathbb{R}, M)$; see Definition \ref{NET space}. By Proposition \ref{equiv norm}, we obtain
		\begin{multline*}
			\|\lambda\mathcal{F}f\|_{N_{q',\infty}(M)} = \sup_{e\in M} \frac{1}{|e|^{1/q}}\left|\int_e \lambda(\xi) \mathcal{F}f(\xi) d\xi\right| \geq   \frac{1}{|e_0|^{1/q}}\left|\int_{e_0} \lambda(\xi) \mathcal{F}f(\xi) d\xi\right|  \\
			= \frac{1}{|e_0|^{1/q}}\left|\int_{e_0} \lambda(\xi) d\xi\right|,
		\end{multline*}
		so that
		\begin{equation}\label{Tlfb}
		\|T_{\lambda}f\|_{L_{q,\tau}(\mathbb{R})}\gtrsim \frac{1}{|e_0|^{1/q}}\left|\int_{e_0} \lambda(\xi) d\xi\right|.
		\end{equation}
		Since $\chi_{e_0}$ is a monotone even function (modulo shifting), by Theorem 2.2 in \cite{Sagher1976}, we obtain
		\begin{equation*}
		\|f\|_{L_p(\mathbb{R})} \approx \|\chi_{e_0}\|_{L_{p',p}(\mathbb{R})} = \left(\int_{0}^{\infty} \left(t^{\frac{1}{p'}}\chi_{e_0}^*(t)\right)^{p}\frac{dt}{t} \right)^{\frac{1}{p}} = \left(\int_{0}^{|e_0|} t^{\frac{p}{p'}-1}dt\right)^{\frac{1}{p}} = |e_0|^{\frac{1}{p'}}.
		\end{equation*}
		Therefore, \eqref{Tlfb} implies
		\begin{equation*}
			\frac{1}{|e_0|^{1/r'}} \left| \int_{e_0} \lambda(\xi) d\xi \right| \lesssim  \|T_\lambda\|_{L_p\mapsto L_{q,\tau}}.
		\end{equation*}
		Recalling that this is true for an arbitrary $e_0\in M$, we complete the proof.
	\end{proof}
	Now, we prove similar result, but for Fourier series multipliers:

\begin{theorem}\label{lower bound 2}
	Let $1< p \leq 2 \leq q <\infty$ and $1/r = 1/p - 1/q$. Let $0<\tau\leq \infty$ and $W$ be the set of all finite intervals in $\mathbb{Z}$. Then, for any sequence of complex numbers $\lambda = \{\lambda_k\}_{k\in \mathbb{Z}}$, it follows
	\begin{equation*}
		\sup_{e\in W} \frac{1}{|e|^{1/r'}} \left|  \sum_{k\in e} \lambda_k \right| \lesssim \|T_\lambda\|_{L_p(0,1) \mapsto L_{q,\tau}(0,1)},
	\end{equation*}
\end{theorem}

\begin{proof}
	Let $e_0$ be an arbitrary interval on $\mathbb{Z}$, that is $e_0\in W$. Then we choose $f$ with $f\sim\sum_{k\in\mathbb{Z}}a_ke^{2\pi i k x}$ such that
	\begin{equation*}
	a_k=
	\begin{cases}
	1 & \text{for } k \in e_0,\\
	0 & \text{for } k \notin e_0.
	\end{cases}
	\end{equation*}
	By Theorems \ref{nursultanovs inequality} and \ref{inclusion}, we estimate
	\begin{align*}
		\|T_{\lambda}f\|_{L_{q,\tau}(0,1)} \gtrsim \|\lambda a\|_{n_{q',\tau}(W)} 
		\gtrsim \|\lambda a\|_{n_{q',\infty }(W)},
	\end{align*}
	where $n_{q',\tau}(W) = n_{q',\tau}\left([0,1],W\right)$ and $n_{q',\infty}(W) = n_{q',\tau}\left([0,1],W\right)$; see Definition \ref{NET space}.
	From Proposition \ref{equiv norm}, it follows
	\begin{equation*}
		\|\lambda a\|_{n_{q',\infty }(W)} = \sup_{e\in W} \frac{1}{|e|^{1/q}}\left|\sum_{k\in e}\lambda_ka_k\right| \geq  \frac{1}{|e_0|^{1/q}}\left|\sum_{k\in e_0}\lambda_ka_k\right|. 
	\end{equation*}
	Recalling the choice of the sequence $a = \{a_k\}_{k\in \mathbb{Z}}$, we derive that
	\begin{equation}\label{Tlfse}
	\|T_{\lambda} f\|_{L_{q,\tau}}\geq \frac{1}{|e_0|^{1/q}}\left|\sum_{k\in e_0}\lambda_k\right|.
	\end{equation}
	Since $a$ is non-increasing and vanishing at infinity, Theorem 4 in \cite{Sagher} gives
	\begin{align*}
	\|f\|_{L_p(0,1)} \approx \|a\|_{l_{p',p}(\mathbb{Z})} = \left( \sum_{k=0}^{|e_0|} \left(k^{1/p'}a^*_k\right)^p \frac{1}{k} \right)^{1/p} \approx \left(\sum_{k=0}^{|e_0|} k^{p-2}\right)^{1/p} \approx |e_0|^{1/p'}.
	\end{align*}
	This and \eqref{Tlfse} give 
	\begin{align*}
	\|T_\lambda\|_{L_p \mapsto L_{q,\tau}} \gtrsim \frac{\|T_{\lambda} f\|_{L_{q,\tau}(0,1)}}{\|f\|_{L_p(0,1)}} \gtrsim \frac{\frac{1}{|e_0|^{1/q}}\left|\sum_{m\in e_0}\lambda_m\right|}{|e_0|^{1/p'}} = \frac{1}{|e_0|^{1/r'}}\left|\sum_{m\in e_0}\lambda_m\right|.
	\end{align*}
Since, $e_0$ was an arbitrary interval, this finishes the proof.
\end{proof}

\begin{remark}
	In case $\tau = q$, Theorem \ref{lower bound 2} was obtained in \cite{NursultanovTleukhanova1999}. While Theorem \ref{Nes condition transform} was obtained only for the case of non-negative symbols, see \cite{Nursultanov1997}.
\end{remark}

\section{Criteria for the $L_p-L_q$ boundedness}\label{sec_criteria}
In this section, we introduce the notion of $M$-generalized monotone functions and sequences. For the corresponding Fourier multipliers, we obtain criteria for $L_p\rightarrow L_q$ boundedness.
\begin{definition}
	Let $M$ be a set of all finite intervals on $\mathbb{R}$.  We say that $f:\mathbb{R} \mapsto \mathbb{C}$ is a $M$-generalized monotone function if 
	\begin{equation*}
		f^*(t) \leq C \bar{f}(t,M)
	\end{equation*}
holds for some $C>0$ depending on $f$.

Let $W$ be a set of all finite intervals in $\mathbb{Z}$. We say that a sequence of complex numbers $\{a_k\}_{k\in \mathbb{Z}}$ is $M$-generalized monotone if 
\begin{equation*}
	a^*_k \leq C \bar{a}_k(W).
\end{equation*}
\end{definition}
This is the simplified version of definition needed for the purpose of this work. For a more general setting we define it as follows:
\begin{definition}
	Let $(\Omega, \mu)$ be a measurable space and $M$ be a set of measurable subsets of $\Omega$ with finite positive measures. We say that $f:\Omega \mapsto \mathbb{C}$ is a $M$-generalized monotone function if
	\begin{equation*}
			f^*(t) \leq C \bar{f}(t,M)
	\end{equation*}
holds for some $C>0$ depending on $f$.
\end{definition}

\begin{theorem}\label{Criteria}
	Let $M$ be a set of all finite intervals on $\mathbb{R}$, $1< p \leq 2 \leq q <\infty$, and $1/r = 1/p - 1/q$. Then a $M$-generalized monotone function $\lambda:\mathbb{R}\mapsto\mathbb{C}$ belongs to $M_p^q$ if and only if  
	\begin{equation}\label{Cr1}
		\sup_{e \in M} \frac{1}{|e|^{1/r'}} \left| \int_{e} \lambda(\xi) d\xi \right| < \infty.
	\end{equation}
\end{theorem}

\begin{proof}
	If $\lambda\in M_p^q$, then Theorem \ref{Nes condition transform} gives \eqref{Cr1}. To prove the converse, it suffices to show that the upper bound in Theorem \ref{HT}(i) is finite. To do this, we estimate 
	\begin{align*}
		\sup_{k\in \mathbb{Z}} \sup_{e \subset \Delta_k} \frac{1}{|e|^{1/r'}} \left| \int_{e} \lambda(\xi) d\xi \right| &\leq \sup_{k\in \mathbb{Z}} \sup_{e \subset \Delta_k} \frac{1}{|e|^{1/r'}}  \int_{e} \left|\lambda(\xi)\right| d\xi \lesssim \sup_{k\in \mathbb{Z}} \sup_{e \subset \Delta_k} \frac{1}{|e|^{1/r'}} \int_{0}^{|e|} \lambda^*(t)  dt.
	\end{align*}
Therefore, since $\lambda$ is a $M$-generalized monotone function, we obtain 
\begin{align*}
	\sup_{k\in \mathbb{Z}} \sup_{e \subset \Delta_k} \frac{1}{|e|^{1/r'}} \left| \int_{e} \lambda(\xi) d\xi \right| &\lesssim \sup_{k\in \mathbb{Z}} \sup_{e \subset \Delta_k} \frac{1}{|e|^{1/r'}} \int_{0}^{|e|} \bar{\lambda}(t,M)  dt\\
	& \lesssim \sup_{k\in \mathbb{Z}} \sup_{e \subset \Delta_k} \frac{1}{|e|^{1/r'}} \int_{0}^{|e|} \sup_{e'\in M, |e'|\geq t} \frac{1}{|e'|^{1 - 1/r'}} \frac{1}{|e'|^{1/r'}} \left| \int_{e'} \lambda(\xi)d\xi\right| dt\\
	& \lesssim \sup_{k\in \mathbb{Z}} \sup_{e \subset \Delta_k} \frac{1}{|e|^{1/r'}} \int_{0}^{|e|} \frac{1}{t^{1 - 1/r'}} dt \sup_{e'\in M} \frac{1}{|e'|^{1/r'}} \left| \int_{e'} \lambda(\xi)d\xi\right| \\
	&\lesssim \sup_{e\in M} \frac{1}{|e|^{1/r'}} \left| \int_{e} \lambda(\xi)d\xi\right| \\
	& <\infty.
\end{align*}
Then, by Theorem \ref{Suf condition 1(i)}, it follows that $\lambda\in M_p^q$.
\end{proof}

Similar result holds for Fourier series multipliers:
\begin{theorem}\label{Criteria for series}
	Let $W$ be a set of all finite intervals on $\mathbb{Z}$, $1< p \leq 2 \leq q <\infty$, and $1/r = 1/p - 1/q$. Then a $M$-generalized monotone sequence $\lambda = \{\lambda_k\}_{k\in \mathbb{Z}}$ belongs to $m_p^q$ if and only if
	\begin{equation}\label{Crit ser}
		\sup_{e\in W} \frac{1}{|e|^{1/r'}} \left| \sum_{j\in e} a_j\right| < \infty.
	\end{equation}
\end{theorem}
\begin{proof}
	Due to Theorems \ref{HT} and \ref{lower bound 2}, it is suffices to prove that the upper bound in Theorem \ref{HT} (ii) is finite if \eqref{Crit ser} holds. We estimate
	\begin{align*}
		\sup_{k\in \mathbb{Z}} \sup_{e \subset \delta_k} \frac{1}{|e|^{1/r'}} \left|  \sum_{m\in e} \lambda_m \right| \leq \sup_{k\in \mathbb{Z}} \sup_{e \subset \delta_k} \frac{1}{|e|^{1/r'}}   \sum_{j=1}^{|e|} \lambda_m^*
	\end{align*}
Since $\lambda$ is a $M$-generalized monotone sequence, we obtain
\begin{align*}
	\sup_{k\in \mathbb{Z}} \sup_{e \subset \delta_k} \frac{1}{|e|^{1/r'}} \left|  \sum_{m\in e} \lambda_m \right| &\lesssim \sup_{k\in \mathbb{Z}} \sup_{e \subset \delta_k} \frac{1}{|e|^{1/r'}}  \sum_{j=1}^{|e|} \sup_{e_0\in W, |e_0|\geq j} \frac{1}{|e_0|} \left| \sum_{j\in e_0} \lambda_j \right|\lesssim \sup_{e\in W} \frac{1}{|e|^{1/r'}} \left|  \sum_{m\in e} \lambda_m \right| .\\
\end{align*}
This completes the proof.
\end{proof}

\section{Examples and corollaries}\label{examandcor}
In the final section, as a corollary, we will prove Theorem \ref{HT}. We will demonstrate that our results are strictly stronger thatn H\"ormander's and Lizorkin's multiplier theorems.
\begin{proof}[Proof of Theorem \ref{HT}]
	The upper bounds in Theorem \ref{HT} follows from Theorems \ref{sufHtr} and \ref{Suf condition 1(i)}. Since $M_k\subset M$ and $W_k\subset W$, by choosing $\tau = q$ in Theorems \ref{Nes condition transform} and \ref{lower bound 2}, we obtain the lower bounds.
\end{proof}

Next, we obtain the following known result.
\begin{corollary}
	\begin{enumerate}[label=(\roman*)]
		\item Let $\lambda$ be a measurable function on $\mathbb{R}$, then 
		\begin{equation*}
			\|\lambda\|_{M_2^2} \approx \|\lambda\|_{L_{\infty}(\mathbb{R})},
		\end{equation*}
	that is $\lambda \in M_2^2$ if and only if $\lambda\in L_\infty(\mathbb{R})$.
	\item Let $\lambda$ be a sequence of complex numbers, then 
	\begin{equation*}
		\|\lambda\|_{m_2^2} \approx \|\lambda\|_{l_{\infty}(\mathbb{Z})},
	\end{equation*}
	that is $\lambda\in m_2^2$ if and only if $\lambda\in l_\infty(\mathbb{Z})$.
	\end{enumerate}
\end{corollary}
\begin{proof}
	The first part follows from Theorems \ref{sufHtr} and \ref{Nes condition transform}, while the second one follows from Theorems \ref{Suf condition 1(i)} and \ref{lower bound 2}.
\end{proof}

For the Fourier series multipliers, we also have the following result:

\begin{corollary}\label{tau to tau}
	Let $1<\tau <\infty$, then
	\begin{equation}\label{tau to tau 2}
		\|\lambda\|_{m_\tau^\tau} \lesssim \sup_{k\in \mathbb{N}_0} \|\lambda\|_{l_{\frac{2\tau}{|2 - \tau|},\infty}(\delta_k)}
	\end{equation}
	for a sequence of complex numbers $\{\lambda_k\}_{k\in \mathbb{Z}}$.
\end{corollary}
	\begin{proof}
		The statement follows from Theorem \ref{Suf condition 1(i)}, by choosing $(p,q) = (2,\tau)$ if $2\leq \tau$, and choosing $(p,q) = (\tau,2)$ if $\tau \leq 2$.
	\end{proof}
	
	Let us note that Corollary \ref{tau to tau} and Marcinkiewicz theorem are not equivalent. For the right-hand side of \eqref{tau to tau 2} to be finite it is necessary that $\lambda\in l_{\infty}$, which is not needed for Marcinkiewicz theorem. Conversely, if we choose $\lambda=\{\lambda_k\}_{k\in \mathbb{Z}}$ such that $\lambda_0 = 0$ and
	\begin{equation*}
		\lambda_{\pm k} =(-1)^{k} \frac{1}{k^{\frac{|\tau - 2|}{2\tau}}} \qquad \text{for } k\in \mathbb{N}.
	\end{equation*}
	Then the right-hand side of \eqref{tau to tau 2} is bounded by $1$, while 
	\begin{equation*}
		\sup_{n\in \mathbb{N}_0}  \sum_{k=2^n}^{2^{n+1}-1} |\lambda_k - \lambda_{k-1}| = \infty.
	\end{equation*}

Next, we show that the sufficient conditions in Theorem \ref{HT} are strictly weaker than \eqref{H_t_cond} and \eqref{H_s_cond}.

Let $\chi_{\Delta_{k}}$ be the indicator function of $\Delta_k$. Then, for the distribution functions of $\lambda\chi_{\Delta_{k}}$ and $\lambda$, it follows that $d_{\lambda\chi_{\Delta_{k}}}(\sigma) \leq d_\lambda(\sigma)$, and hence,
\begin{equation*}
	(\lambda\chi_{\Delta_{k}})^*(t)\leq \lambda^*(t), \qquad \text{for } t>0, \text{ } k\in \mathbb{Z}.
\end{equation*}
Therefore
\begin{equation*}
	\sup_{k\in \mathbb{Z}} \sup_{e \subset \Delta_k} \frac{1}{|e|^{1/r'}} \left| \int_{e} \lambda(\xi) d\xi \right|  \approx \sup_{k\in \mathbb{Z}} \|\lambda\|_{L_{r,\infty}(\Delta_{k})} \leq \|\lambda\|_{L_{r,\infty}(\mathbb{R})},
\end{equation*}
so that Theorem \ref{HT}(i) implies the H\"ormander's theorem for  Fourier transform multipliers. 

Let us consider the following example:
\begin{example}\label{exmH1}
	Let $r>0$ and $\lambda$ be an even function such that
	\begin{equation*}
		\lambda(\xi) = \frac{1}{(\xi - 2^k)^\frac{1}{r}}, \qquad \text{for }\xi\in (2^k,2^{k+1}), \text{ } k\in \mathbb{Z}.
	\end{equation*}
\end{example}
Since $d_\lambda(\sigma) = \infty$ for $0<\sigma<\infty$, we obtain that $\|\lambda\|_{L_{r,\infty}(\mathbb{R})} = \infty$, so that we can not apply H\"ormander's theorem. However, one can check that 
\begin{equation*}
	d_{\lambda\chi_{\Delta_{k}}}(\sigma) = 
	\begin{cases}
		2^k & \sigma \leq \left(\frac{1}{2^k}\right)^{1/r},\\
		\left(\frac{1}{\sigma}\right)^r & \sigma > \left(\frac{1}{2^k}\right)^{1/r},
	\end{cases}
	\qquad
	\left(\lambda\chi_{\Delta_{k}}\right)^*(t) = 
	\begin{cases}
		\left(\frac{1}{t}\right)^{1/r} & t\leq 2^k,\\
		0 & t>2^k.
	\end{cases}
\end{equation*}
Therefore
\begin{equation*}
	\|\lambda\|_{L_{r,\infty}(\Delta_{k})} = \sup_{t>0} t^{1/r} \left(\lambda\chi_{\Delta_{k}}\right)^*(t) = \sup_{0<t\leq 2^k} t^{1/r} t^{-1/r} =1,
\end{equation*}
and hence, by Theorem \ref{HT}(i), it follows that $\lambda\in M_{p}^q$.

Similarly, one can show that
\begin{equation*}
	\sup_{k\in \mathbb{N}_0} \sup_{e \subset \delta_k} \frac{1}{|e|^{1/r'}} \left|  \sum_{m\in e} \lambda_m \right| 
	\approx \sup_{k\in \mathbb{N}_0}\|\lambda\|_{l_{r,\infty}(\delta_k)} 
	\lesssim \|\lambda\|_{l_{r,\infty}},
\end{equation*}
so that Theorem \ref{HT}(ii) implies the H\"ormander theorem for  Fourier series multipliers. 

Let us consider the example:
\begin{example}\label{examH2}
	Let $r>0$ and $\lambda=\{\lambda_k\}_{k\in \mathbb{Z}}$ be the sequence such that
	\begin{equation*}
		\lambda_j
		=
		\begin{cases}
			\left(\frac{1}{j + 1 - 2^k}\right)^{\frac{1}{r}}, & \text{for } j\in \delta_{k} \text{ and } k\in \mathbb{N},\\
			0 ,& j\leq 0.
		\end{cases}
	\end{equation*}
	Since $\lambda_j^* = 1$ for $j\in \mathbb{N}$, we derive
	\begin{equation*}
		\|\lambda\|_{l_{r,\infty}(\mathbb{Z})} = \sup_{k\geq 0}k^{1/r} \lambda_k^*=\infty,
	\end{equation*}
	while
	\begin{equation*}
		\|\lambda\|_{l_{r,\infty}(\delta_k)} = 1 <\infty.
	\end{equation*}
	Therefore, we can not apply the H\"ormander's theorem, however, we can apply Theorem \ref{HT}(ii) to see that $\lambda\in m_p^q$.
\end{example}

Further, we check that the sufficient conditions in Theorem \ref{LT} are strictly weaker than \eqref{L_t_cond} and \eqref{L_s_cond}. Indeed, we estimate
\begin{multline*}
	2^{\frac{k}{r}} \int_{\Delta_{k}} |\lambda'(\xi)| d\xi = 2^{\frac{k}{r}} \int_{\Delta_{k}} |\lambda'(\xi)| |\xi|^{\frac{1}{r}+1} |\xi|^{-\frac{1}{r}-1} d\xi 
	\leq \sup_{\xi\in \mathbb{R}} |\lambda'(\xi)| |\xi|^{\frac{1}{r}+1} 2^{\frac{k}{r}} \int_{\Delta_{k}} |\xi|^{-\frac{1}{r}-1} d\xi \\
	\leq \sup_{\xi\in \mathbb{R}} |\lambda'(\xi)| |\xi|^{\frac{1}{r}+1} r  \left(1 - 2^{-\frac{1}{r}}\right).
\end{multline*}
Therefore, Theorem \ref{LT}(i) gives \eqref{L_t_cond}. 

We consider the following example:
\begin{example}\label{examL1}
	Let $0<\alpha<1$ and $\lambda$ be an even function on $\mathbb{R}$ such that
	\begin{equation*}
		\lambda(\xi) = 
		\begin{cases}
			(2 - x)^\alpha & x\leq 2,\\
			0 & x>2.
		\end{cases}
	\end{equation*}
	Note that $\lambda'(\xi) = \alpha(2 - x)^{\alpha - 1}$ on $[0,2)$, which is not bounded at $\xi = 2$. Therefore, the 
	right-hand side of \eqref{L_t_cond} is infinite.  However, since the singularity of $\lambda'$ at $\xi = 2$ is integrable, we conclude that 
	\begin{equation*}
		 \sup_{k\in \mathbb{Z}} 2^{\frac{k}{r}} \int_{\Delta_k} |\lambda'(\xi)| d\xi <\infty.
	\end{equation*}
	Moreover, $\lambda$ is absolutely continuous on $(-\infty, 0]$ and $[0,\infty)$, and $|\lambda(\xi)|\rightarrow0$ as $|\xi|\rightarrow\infty$. Therefore, by Theorem \ref{LT} (i), $\lambda\in M_p^q$ for $1<p<q<\infty$.
\end{example}

We repeat these arguments for the second part. We write
\begin{multline*}
	2^{\frac{k}{r}}\sum_{m=2^k}^{2^{k+1} - 1} \left(|\lambda_{ -m} - \lambda_{ -m +1}| + |\lambda_{ m} - \lambda_{ m -1}|\right)\\
	= 2^{\frac{k}{r}}\sum_{m=2^k}^{2^{k+1} - 1} \left(|\lambda_{ -m} - \lambda_{ -m +1}| + |\lambda_{ m} - \lambda_{ m -1}|\right)|m-1|^{1 + \frac{1}{r}} |m-1|^{-1 - \frac{1}{r}},
\end{multline*}
so that 
\begin{multline*}
	2^{\frac{k}{r}}\sum_{m=2^k}^{2^{k+1} - 1} \left(|\lambda_{ -m} - \lambda_{ -m +1}| + |\lambda_{ m} - \lambda_{ m -1}|\right)\\
	\lesssim \sup_{n\in \mathbb{Z}} |\lambda_n - \lambda_{n+1}|n^{1 + \frac{1}{r}} 2^{\frac{k}{r}}\sum_{m=2^k}^{2^{k+1} - 1}|m-1|^{-1 - \frac{1}{r}} \lesssim \sup_{n\in \mathbb{Z}} |\lambda_n - \lambda_{n+1}|n^{1 + \frac{1}{r}}.
\end{multline*}

The following example demonstrates that the converse inequality does not hold:
\begin{example}\label{examL2}
	Let $1<p<q<\infty$, $1/r = 1/p - 1/q$, and 
	\begin{equation*}
		\gamma = \sum_{j=0}^{\infty} \left(\frac{1}{2^{1/r}}\right)^{j}.
	\end{equation*}
	Then, we define recursively 
	\begin{align*}
		&\lambda_0 = \gamma,\\
		&\lambda_{2^k}=\cdots=\lambda_{2^{k+1}-1}=\lambda_{2^k-1} - \left(\frac{1}{2^{1/r}}\right)^k
	\end{align*}
	for $k\in \mathbb{N}_0$. We also set $\lambda_{-k} = \lambda_{k}$ and compute
	\begin{equation*}
		\left(2^k - 1\right)^{\frac{1}{r} + 1}\left|\lambda_{2^k-1} - \lambda_{2^k}\right| = 	\left(2^k - 1\right)^{\frac{1}{r} + 1}\left(\frac{1}{2^{1/r}}\right)^k = \left(\frac{2^k - 1}{2^k}\right)^{1/r} \left(2^k - 1\right)
	\end{equation*}
	which is unbounded as $k\rightarrow\infty$. Hence, the Lizorkin's theorem is not applicable. However, we can apply Theorem \ref{LT}. Indeed, by definition of $\gamma$ and $\lambda_j$, we obtain
	\begin{equation*}
		\lambda_{2^k}=\cdots=\lambda_{2^{k+1}-1}=\gamma - \sum_{j=0}^{k}\left(\frac{1}{2^{1/r}}\right)^j \rightarrow 0
	\end{equation*}
	as $k\rightarrow \infty$. Further, we compute
	\begin{equation*}
		2^{\frac{k}{r}}\sum_{j=2^k}^{2^{k+1}-1} \left| \lambda_j - \lambda_{j - 1}\right| = 2^{\frac{k}{r}} \left| \lambda_{2^{k}} - \lambda_{2^{k} - 1}\right| = 2^{\frac{k}{r}}\left(\frac{1}{2^{1/r}}\right)^{k} =1.
	\end{equation*}
	Therefore, by Theorem \ref{LT}, it follows that $\lambda\in m_p^q$.
\end{example}

Finally, we note that Theorem \ref{LT} is at least complementary to \cite[Theorem 2]{SarybekovaTararykovaTleukhanova} and \cite[Theorem 1.3]{PerssonSarybekovaTleukhanova}, respectively. To see this, consider the following examples.

\begin{example}\label{Laz1}
	For $0<\gamma< 1$, define a function
	\begin{equation}
		\lambda(x):=
		\begin{cases}
			2^{-\frac{k}{r}} \left( 2 - |(2^k+2) - x| \right)^\gamma & x\in [2^k, 2^k + 4] \text{ and } k\geq2,\\
			0 & \text{otherwise}.
		\end{cases}
	\end{equation}
\end{example}
First, we check that 
\begin{equation*}
	\sup_{k\in \mathbb{Z}}2^{\frac{k}{r}} \int_{\Delta_{k}} |\lambda'(x)|dx = \sup_{k\in \mathbb{Z}} 2\int_{2^k+2}^{2^k+4} (2^k+4-x)^{\gamma - 1}dx < \infty.
\end{equation*}
Therefore, by Theorem \ref{LT}, $\lambda$ represents $L_p\rightarrow L_q$ Fourier multiplier.

Further, for $\alpha<1 - 1/r$ and $\beta = \alpha + 1/r$, the inequality holds
$$|\lambda'(x)x^\beta| > g(x),$$
where
\begin{equation*}
	g(x) =
	\begin{cases}
		\gamma 2^{-\frac{k}{r}} \left( 2 - |(2^k+2) - x| \right)^{\gamma - 1}2^{k\beta},& x\in [2^k, 2^k + 4] \text{ and } k\geq2,\\
		0, & \text{otherwise}.
	\end{cases}
\end{equation*}
Note that $2^{-k/r}2^{\beta k} = 2^{\alpha k}$. One can check that $g^*=\infty$, therefore, $\lambda$ does not satisfy conditions of \cite[Theorem 2]{SarybekovaTararykovaTleukhanova}.

Consider the following example:
\begin{example}\label{Laz2}
	Let
	\begin{equation*}
		\lambda_{m}=
		\begin{cases}
			2^{-\frac{k}{r}}, & m=2^{k}+1 \text{ and } k\geq 2,\\
			0, & \text{otherwise.}
		\end{cases}
	\end{equation*}
\end{example}
Then
\begin{equation*}
	\sup_{k\in \mathbb{N}_0} 2^{\frac{k}{r}} \sum_{m=2^k}^{2^{k+1} - 1} |\lambda_m - \lambda_{m-1}|= 2<\infty.
\end{equation*}
Therefore, by Theorem \ref{LT}, $\lambda$ represents $L_p\rightarrow L_q$ Fourier multiplier. However, this can not be seen from \cite[Theorem 1.3]{PerssonSarybekovaTleukhanova}. Indeed, let $\alpha<1 - 1/r$ and $\beta = \alpha + 1/r$. Denote
\begin{equation*}
	\eta_k: = k^\beta |\lambda_k - \lambda_{k+1}|.
\end{equation*}
One can check that $\eta_{2^k} = 2^{\alpha k}$ for $k\geq 2$, therefore, $\eta_k^* = \infty$, so that $\lambda$ dose not satisfies condition of \cite[Theorem 1.3]{PerssonSarybekovaTleukhanova}.
\begin{remark}
	Note that \cite[Theorem 2]{SarybekovaTararykovaTleukhanova} and \cite[Theorem 1.3]{PerssonSarybekovaTleukhanova} are stronger than corresponding H\"ormander theorems. In particular, Examples \ref{Laz1} and \ref{Laz2} do not satisfy \eqref{H_t_cond} and \eqref{H_s_cond}, respectively.
\end{remark}

\bibliographystyle{plain}
\bibliography{references}

\setlength{\parskip}{0pt}





\end{document}